\documentclass[11pt]{article}
 \usepackage[margin=1in]{geometry} 

\usepackage[utf8]{inputenc} 
\usepackage[T1]{fontenc}    
\usepackage{hyperref}       
\usepackage{url}            
\usepackage{booktabs}       
\usepackage{amsfonts}       
\usepackage{nicefrac}       
\usepackage{microtype}      

\usepackage[round]{natbib}
\usepackage{amsmath,amsfonts}
\usepackage{graphicx, caption, subcaption}
\usepackage{amsthm}
\usepackage[ruled,linesnumbered]{algorithm2e}
\usepackage{xcolor}
\usepackage{todonotes}
\newtheorem{theorem}{Theorem}[section]

\newtheorem{lemma}[theorem]{Lemma}
\newtheorem{prop}[theorem]{Proposition}

\newtheorem{definition}[theorem]{Definition}
\newtheorem{remark}[theorem]{Remark}
\newcommand{\RR}{\mathbb{R}}
\newcommand{\dz}{\dot{z}}

\newcommand{\twomatrix}[4]{\left[\begin{matrix}  #1 & #2\\#3  & #4 \end{matrix}\right]}

\newcommand{\pran}[1]{\left(#1\right)}

\newcommand{\vectorr}[2]{\begin{bmatrix}  #1 \\ #2 \end{bmatrix}}
\newcommand{\uP}{\bar{P}}
\newcommand{\lP}{\underbar{$P$}}

\begin{document}
\title{Limiting Behaviors of Nonconvex-Nonconcave Minimax Optimization via Continuous-Time Systems}
\author{Benjamin Grimmer\footnote{bdg79@cornell.edu; Cornell University, Ithaca NY}, Haihao Lu\footnote{Haihao.Lu@chicagobooth.edu; Google Research, New York NY and University of Chicago, Chicago IL}, Pratik Worah\footnote{pworah@google.com; Google Research, New York NY}, Vahab Mirrokni\footnote{mirrokni@google.com;Google Research, New York NY}}
\date{}
\maketitle

\begin{abstract}
  Unlike nonconvex optimization, where gradient descent is guaranteed to converge to a local optimizer, algorithms for nonconvex-nonconcave minimax optimization can have topologically different solution paths: sometimes converging to a solution, sometimes never converging and instead following a limit cycle, and sometimes diverging. In this paper, we study the limiting behaviors of three classic minimax algorithms: gradient descent ascent (GDA), alternating gradient descent ascent (AGDA), and the extragradient method (EGM). Numerically, we observe that all of these limiting behaviors can arise in Generative Adversarial Networks (GAN) training and are easily demonstrated for a range of GAN problems. To explain these different behaviors, we study the high-order resolution continuous-time dynamics that correspond to each algorithm, which results in the sufficient (and almost necessary) conditions for the local convergence by each method. Moreover, this ODE perspective allows us to characterize the phase transition between these different limiting behaviors caused by introducing regularization as Hopf Bifurcations. 
\end{abstract}

\section{Introduction}
In this paper, we are interested in the limiting behavior of optimizing nonconvex-nonconcave problems 
\begin{equation} \label{eq:minimax}
    \min_{x\in\RR^n}\max_{y\in\RR^m} L(x,y) ,
\end{equation}
for any differentiable objective function $L(x,y)$.
Minimax optimization has found wide usage in robust optimization~\citep{verdu1984minimax, ben2009robust, bertsimas2011theory} and many machine learning tasks. One notable application is in GAN training~\citep{Goodfellow1412} with a generator trying to produce fake data samples $G$ from a latent distribution and a discriminator $D$ trying to distinguish these from true data, defined by the following minimax problem
\begin{align} 
    \min_G \max_D\ &\mathbb{E}_{s\sim p_{data}} \left[\log D(s)\right] \nonumber\\
    & + \mathbb{E}_{e\sim p_{latent}} \left[\log(1-D(G(e)))\right]. \label{eq:GAN} 
\end{align}

We consider first-order methods for solving the minimax problem~\eqref{eq:minimax} given a gradient oracle $F(x,y)=(\nabla_x L(x,y), -\nabla_y L(x,y))$ or an unbiased estimator of these gradients. Given such an oracle, three of the most classic first-order minimax optimization methods can be defined as follows, producing a sequence of solution pairs $z_k=(x_k,y_k)$:\\
Gradient Descent Ascent (GDA)
\begin{equation}\label{eq:GDA}
    z_{k+1} = z_{k} - s F(z_{k}),
\end{equation}
Alternating Gradient Descent Ascent (AGDA)
\begin{align}\label{eq:AGDA}
    x_{k+1} &= x_{k} - s F_x(x_{k},y_{k}) \nonumber \\
    y_{k+1} &= y_{k} - s F_y(x_{k+1},y_{k}),
\end{align}
and the Extragradient Method (EGM)
\begin{align}\label{eq:EGM}
    z_{k+1/2} &= z_{k} - s F(z_{k}) \nonumber \\
    z_{k+1} &= z_{k} - s F(z_{k+1/2}).
\end{align}
Stochastic versions of these algorithms all follow by replacing $F(x,y)$ with an unbiased estimator. In the case of GAN training, an unbiased gradient estimate is given by using a finite batch of samples from $p_{data}$ and $p_{latent}$ to approximate the expectations in~\eqref{eq:GAN}.

\begin{figure*}\centering
    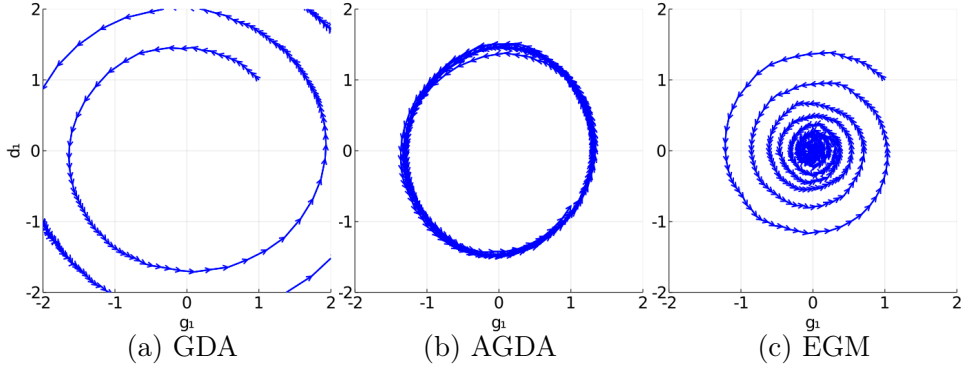
    \caption{Sample trajectories of GDA, AGDA, and EGM with batch size $10$ applied to a simple GAN~\eqref{eq:GAN} showing the first coordinates $(g_1,d_1)$ of $g$ and $d$ which diverge, cycle, and converge around a stationary point at the origin, respectively.}
    \label{fig:cifar}
\end{figure*}

We begin by considering the limiting behavior of these three methods on a GAN training example over CIFAR-10 data (with full experiment details given in Section~\ref{sec:GAN}). We take simple models for the discriminator controlling weights $d\in\RR^n$ in the logistic function
\begin{equation}\label{eq:discriminator}
    D(s) = \frac{1}{1+e^{d^T(s-\bar s)}}
\end{equation}
where $\bar s = \mathbb{E}_{s\sim p_{data}}[s]$ is the average image over the dataset and the generator controlling a translation of a normal latent distribution $g\in\RR^n$ giving $G(e)= g+e$.
Figure~\ref{fig:cifar} shows that even in this simplified setting the trajectories of GDA, AGDA, and EGM vary greatly, diverging, cycling, and converging respectively.

There has been a recent surge in using continuous time ODE models to understand the behavior of iterative optimization methods, initiated by~\citet{Su1601}.
A typical focus is on the ODE arising from taking the stepsize $s\rightarrow 0$ to zero. However, under this limit, all three of these methods have their solution paths converge to the {\it gradient flow} (GF) given by the ODE
\begin{equation}\label{eq:ODE-GF}
    \dot z = -F(z).
\end{equation}
To understand the differences in limiting behavior we have observed between GDA, AGDA, and EGM, we need to consider the $O(s)$-resolution ODEs proposed by~\citet{Shi1810} and~\citet{Lu2020}, which capture differences between these methods when the stepsize $s$ is nonzero.

\subsection{Our Contributions}
The main contributions of this work is understanding and characterizing different possible (potentially nonconvergent) limiting behaviors and limit points of minimax algorithms, which enables us to explain the differing trajectories observed in our one-layer GAN example.

\begin{enumerate}
\item {\bf Divergence, cycling and limit points.}
We derive necessary and sufficient conditions for stationary points to be attractive for each of GDA, AGDA, and EGM of the generic minimax problem~\eqref{eq:minimax}. These conditions apply broadly to any sufficiently differentiable nonconvex-nonconcave minimax problems and explain the differences in convergence and divergence found in our GAN experiments. Our conditions are based on understanding the underlying ODEs, but apply directly to each discrete-time algorithms.

\item {\bf Regularization induced phase transitions.} Adding regularization into our our one-layer CIFAR example eventually leads GDA and AGDA to converge. 
Figure~\ref{fig:GDA-regularization} shows GDA transitions from divergence to having an attractive limit cycle, which then shrinks eventually collapsing into an attractive stationary point.\\
For a broad class of GAN training problems, we show this transition from a nonconvergent trajectory to convergence will be a Hopf bifurcation. We present concrete two-dimensional GAN and polynomial minimax examples demonstrating Hopf bifurcations occurring.
\end{enumerate}

\subsection{Related Works}

\paragraph{Divergence and cycling.}
GDA is well known to diverge even for convex-concave problems while EGM still converges, the simplest example being $\min_{x\in\RR}\max_{y\in\RR} xy$. Cycling behaviors arising in nonconvex-nonconcave problems have been observed in a variety of different settings~\citep{Letcher2005,Hsieh2006,Grimmer2006}. We differ from these works in that our focus is on developing tools for characterizing attractive limit points and the transitions between different limiting behaviors for a wide range of minimax algorithms.

\paragraph{\bf GAN equilibria.}
\citet{Arora1710} showed under moderate size requirements, GAN training will have an approximately pure equilibrium but that equilibrium point may be far from target true data distribution.
These results however do not guarantee that an equilibrium point will be found, leaving the potential for cycling and divergence as seen in Figure~\ref{fig:cifar}.

\paragraph{Nonconvex-nonconcave limit points.}
Quantifying limit points as local Nash equilibrium with measures like stationarity $\nabla L(x,y)\approx 0$ (which is the first-order necessary condition for a Nash equilibrium) has been done for a variety of different first-order methods~\citep{Cherukuri1611,Constantinos1812,adolphs1901,Mazumdar2002,liang2019interaction}. We follow in this vein as the limit points of GDA and AGDA necessarily have $\nabla L(x,y)=0$. 

Alternative measures of local optimality are discussed in~\citet{Jordan2008}, where a notion of local minimax points is presented capturing the sequential nature of minimax problems and relating to GDA's limit points.

\paragraph{Stability of Limit Points}
Recently~\citep{Lu2020} presented an ODE approach to analyzing a broad class of discrete-time algorithms. Their analysis is non-trivial and requires high-order smoothness (five times differentiability). Our result differ from theirs as they require strong global conditions, whereas we just need similar conditions hold at the stationary point $z^*$.
As an alternative to this ODE approach, \citep{Zhang2020} study discrete-time algorithms directly, and present conditions of stability that involves the complex eigen-values of the non-symmetric Jacobian matrix, which cannot be easily verified even for simple problems. In contrast, our approach leads to much more transparent conditions on a p.s.d. matrix which facilitates our study on the role of the stepsize and  derivatives in convergence, as well as the usage of insights from ODE and Bifurcation theory. For example, our result clearly shows that interaction terms help the convergence of EGM while hurt GDA.

\paragraph{Nonconvex-nonconcave convergence rates.}
Beyond characterizing the limit points, there has also been a recent push to establish finite-time convergence guarantees for nonconvex-nonconcave minimax problems. Many of these approaches rely on forms of convex-concave-like assumptions, such as Minty's Variational Inequality~\citep{Lin1809} and Polyak-Lojasiewicz conditions~\citep{Nouiehed1902, Yang2002}.
Convergence guarantees for the proximal point method on quite general nonconvex-nonconcave problems are given in~\citep{Grimmer2006} when the interaction between the minimizing and maximizing agents is sufficiently strong or sufficiently weak. The limitation to these convex-concave like assumptions or focusing on the proximal point method however prevents these results from describing the phenomena observed herein for GANs.


\section{GAN Divergence, Cycling and Phase Transitions} \label{sec:GAN}
As briefly described in the introduction, we begin by surveying the types of solution paths that arise numerically from GAN training~\eqref{eq:GAN}.
We fix the true data distribution $p_{data}$ as being uniformly over the set of $50,000$ training images in the CIFAR-10 dataset, each represented as vectors of length $n=32\times32\times3$ and the latent distribution $p_{latent}$ as a standard normal.

We find that considering one-layer networks already suffices to encounter a wide range of different solution path geometries when solving~\eqref{eq:GAN}. 
We consider a discriminator controlling weights $d\in\RR^n$ in the logistic~\eqref{eq:discriminator} and a generator controlling a translation $g\in\RR^n$ giving $G(e)= g+e$.

For this setup, we find that the origin $(g,d)=(0,0)$ has $F(0,0)\approx (0,0)$ and so it is an approximate stationary point for all three of GDA, AGDA, and EGM. However, the solution paths of these methods vary widely when initialized near the origin at $g=(1,0,...,0)$ and $d=(1,0,...,0)$ with fixed $s=0.2$, not all converging to the origin. Sample solution paths using batches of $10$ samples to approximate the gradients are shown in Figure~\ref{fig:cifar}, plotting the first coordinate of $g$ and $d$ which had the only nonzero initializations. The appendix gives more sample trajectories using other batch sizes, showing the same general dynamics hold but either more or less noisy.

This experiment shows three different limiting behaviors arising: GDA diverges spiraling outward from the origin, AGDA falls into a stable cycle around the origin, and EGM converges into a stationary point around the origin. In Section~\ref{sec:limits}, we develop theory based on related $O(s)$-resolution ODE models explaining these observed differences in convergence and divergence.

Since GDA and AGDA failed to converge, a typical recourse is to add regularization to the objective to deter this behavior. Using L2 penalization, we have the modified training problem objective
\begin{align} \label{eq:regularized}
    \mathbb{E}_{s\sim p_{data}} \left[\log D(s)\right] &+ \mathbb{E}_{e\sim p_{latent}} \left[\log(1-D(G(e)))\right]\nonumber \\
   & + \frac{\alpha}{2}\|g\|^2 - \frac{\alpha}{2}\|d\|^2, 
\end{align} 
for any level of regularization $\alpha\geq 0$. Adding mild amounts of regularization $\alpha<1.0$ suffices to cause the nonconvergent trajectories of GDA and AGDA to transition into being convergent paths.

\begin{figure*}\centering
    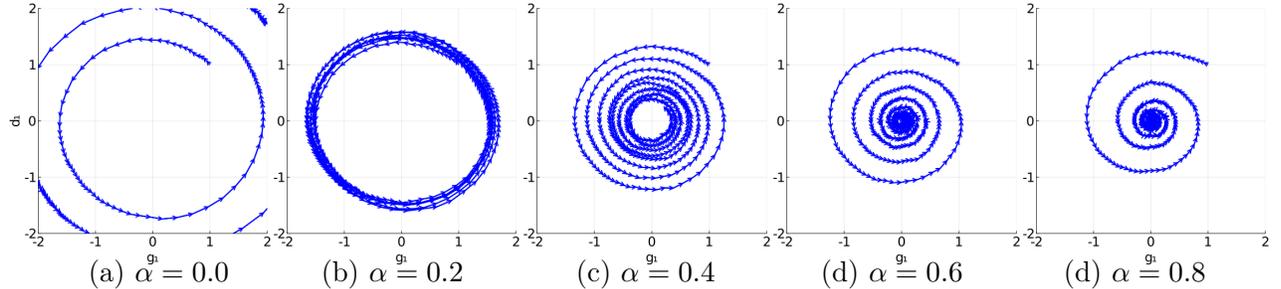
    \caption{GDA trajectory on a simple GAN with increasing quadratic regularization, plotting the first coordinates $(g_1,d_1)$ at each iteration.}
    \label{fig:GDA-regularization}
\end{figure*}

Figure~\ref{fig:GDA-regularization} shows the trajectory of GDA~\eqref{eq:AGDA} as the regularization parameter increases. We see that as $\alpha$ grows larger, the stable limit cycle shrinks until it collapses to a point for some critical value of $\alpha$ between $\alpha=0.4$ and $0.6$. For all $\alpha$ larger than this critical point, the trajectory of GDA has transitioned from nonconvergent cycling to convergence. 


In Section~\ref{sec:hopf}, we explain this phase transitions from having a repulsive stationary point at the origin to an attractive one as a bifurcation occurring in the related $O(s)$-resolution ODE for each algorithm.

\section{The Attractive Limit Points of the Four Dynamics} \label{sec:limits}
Our simple GAN experiments show the trajectories of GDA, AGDA, and EGM can be topologically different. Here we explain such differences by studying the $O(s)$-resolution ODEs for the three algorithms (which are formally defined in Section~\ref{sec:ODEs}) as well as the simple gradient flow dynamics. From these, we arrive at our main theoretical result Theorem~\ref{thm:general_sufficient_condition} in Section~\ref{sec:local_ODE}, giving necessary and sufficient conditions for a stationary point to be attractive for these dynamics. Further, Section~\ref{sec:discrete} shows that for reasonably small $s$, these conditions guarantee the discrete-time iterations have a linear attractor whenever their associated ODE does.

For any positive definite $P\succ 0$, let $\|z\|^2_P=z^TPz$. Our theory applies to an arbitrary stationary point $z^*$ where we denote the second order derivatives of the objective by $A=\nabla^2_{xx} L(z^*), B=\nabla^2_{xy} L(z^*), C=-\nabla^2_{yy} L(z^*)$. 

\subsection{$O(s)$-Resolution ODE Systems for Discrete-time Algorithms} \label{sec:ODEs}
The traditional way to obtain an ODE for a discrete-time algorithm is to let the stepsize $s$ go to $0$. While that may be the easiest and the most natural approach, we never use $s=0$ in practice, and even worse, the solution trajectory of the resulting ODE and the discrete-time algorithm with any positive stepsize can be topologically different. This is exemplified by the trajectories of GDA, AGDA, and EGM seen in Figure \ref{fig:GDA-regularization} which are topologically different, despite sharing the same ODE when letting the stepsize go to $0$, namely gradient flow. In order to overcome such limitations, high-order resolution ODEs~\citep{Shi1810,Lu2020} have been proposed recently to study the subtle differences between similar algorithms. We here utilize the framework proposed in~\citet{Lu2020} to study the $O(s)$-resolution ODE for these three algorithms, which, as we will show later, is capable to explain the different behaviors of the algorithms.

More formally, we say an ODE 
\begin{equation}\label{eq:osODE}
\dz=f_0(z)+s f_1(z)
\end{equation}
the $O(s)$-resolution ODE of a discrete-time algorithm with iterate update $z^{+}=g(z,s)$ if it satisfies that for any $z$ and $z^+=g(z,s)$,
\begin{equation}\label{eq:small_term}
    \|z(s)-z^+\| = o(s^{2}) \ .
\end{equation}
For stochastic gradients, this can be extended by considering the fluid limit as $z^+=\mathbb{E}\  g(z,s)$ given by using large enough batches.
It turns out such an ODE \eqref{eq:osODE} is unique for a smooth enough discrete-time algorithm (see Theorem 1 in~\citet{Lu2020}) and we herein study the $O(s)$-resolution ODE of GDA, AGDA and EGM:

\begin{prop}\label{prop:ODEs}
    1. \textbf{GDA:} The $O(s)$-resolution ODE of GDA is
    \begin{equation}\label{eq:ODE-GDA}
        \dot z = -F(z) - \frac{s}{2}\nabla F(z) F(z).
    \end{equation}
    2. \textbf{EGM:} The $O(s)$-resolution ODE of EGM is
    \begin{equation}\label{eq:ODE-PPM}
        \dot z = -F(z) + \frac{s}{2}\nabla F(z) F(z).
    \end{equation}
    3. \textbf{AGDA:} The $O(s)$-resolution ODE of AGDA is
    \begin{equation}\label{eq:ODE-AGDA}
        \dot z = -F(z) - \frac{s}{2}\twomatrix{A}{B^T}{B}{C} F(z) \ ,
    \end{equation}
\end{prop}


Letting the stepsize go to zero, all three of these ODEs reduce to the previously stated gradient flow ODE~\eqref{eq:ODE-GF}.




\subsection{Local Convergence of the $O(s)$-Resolution ODE}\label{sec:local_ODE}
Armed with these more refined ODEs for GDA, AGDA, and EGM, we can formalize our necessary and sufficient conditions for a limit point to be attractive. For a generic ODE
\begin{equation} \label{eq:general_ODE}
\dz = G(z) \ ,    
\end{equation}
we need conditions for when a fixed point $z^*$ to the ODE~\eqref{eq:general_ODE} is attractive, which will translate into our required conditions for the convergence of specific algorithms. The traditional stability theory of ODE says that the attractivity of a fixed point relies on the whether the real part of the every eigenvalue of $\nabla G(z^*)$ is negative, but unfortunately it is not transparent to directly evaluate the eigenvalue structure of $\nabla G(z^*)$ for minimax problems. Instead, in order to take advantage of the structure of minimax problems, we here study whether $\|z(t)-z^*\|_P$ goes to $0$ for a positive definite norm matrix $P$. The natural norm $P$ to use differs for different algorithms.



\begin{definition}
We say a fixed point $z^*$ 
is a linear attractor to an ODE system \eqref{eq:general_ODE} in the $P$ norm if there exists $\delta, \rho>0$ such that for any initial solution $z(0)\in B(z^*, \delta)$, it holds that $\|z(t)-z^*\|_P\le e^{-\rho t}\|z(0)-z^*\|_P$.\footnote{In ODE literature, this convergence rate sometimes is called exponential convergence. Here we choose to call it linear in order to be consistent with the optimization literature.}
\end{definition}
Then the linear attractors of our ODEs of interest are characterized as follows.
\begin{theorem}\label{thm:main}
1. \textbf{GF:} Suppose $z^*$ is a stationary point to GF \eqref{eq:ODE-GF}. Then $z^*$ is a linear attractor to GF \eqref{eq:ODE-GF} in Euclidean norm (that is, $P=I$) if it holds that 
\begin{equation}\label{eq:sufficient_condition_GF}
    A \succ 0\ , C \prec 0 \ 
\end{equation}
and not if either inequality is strictly violated.\\
2. \textbf{GDA:} Suppose $z^*$ is a stationary point to the $O(s)$-resolution ODE to GDA \eqref{eq:ODE-GDA}. Then $z^*$ is a linear attractor to \eqref{eq:ODE-GDA} in Euclidean norm if 
\begin{equation}\label{eq:sufficient_condition_GDA}
    A+\frac{s}{2}(A^2-BB^T) \succ 0\ , C+\frac{s}{2}(C^2-B^TB) \succ 0 \ 
\end{equation}
and not if either inequality is strictly violated.\\
3. \textbf{EGM:} Suppose $z^*$ is a stationary point to the $O(s)$-resolution ODE to EGM \eqref{eq:ODE-PPM}. Then $z^*$ is a linear attractor to \eqref{eq:ODE-PPM} in Euclidean norm if 
\begin{equation}\label{eq:sufficient_condition_PPM}
    A-\frac{s}{2}(A^2-BB^T) \succ 0\ , C-\frac{s}{2}(C^2-B^TB) \succ 0 \ 
\end{equation}
and not if either inequality is strictly violated.\\
4. \textbf{AGDA:} Suppose $z^*$ is a stationary point to the $O(s)$-resolution ODE to AGDA \eqref{eq:ODE-AGDA}. Then $z^*$ is a linear attractor to \eqref{eq:ODE-AGDA} in scaled Euclidean norm with $P=\twomatrix{I}{\frac{1}{2} sB^T}{\frac{1}{2} sB}{I}$ if
\begin{equation} \label{eq:sufficient_condition_AGDA} 
    M\succ 0
\end{equation}
where $M$ is defined as
$$ M_{xx}=-A-\frac{s}{2}A^2 - \frac{s^2}{4}\pran{A B^T B + B^T B A},$$
$$ M_{yy}=-C-\frac{s}{2}C^2 - \frac{s^2}{4}\pran{C B B^T + B B^T C},$$
$$ M_{xy}=-\frac{s}{2}\pran{BA+CB} - \frac{s^2}{4}\pran{BA^2 + C^2 B},$$
and not if the inequality is strictly violated.
\end{theorem}
Similar sufficient conditions for GDA and EGM were presented in \citet{Lu2020}. Theorem \ref{thm:main} improves upon those by only requiring the positive definiteness conditions to hold only at $z^*$, whereas \citet{Lu2020} requires similar positive definiteness conditions to hold globally. This difference comes from the analysis utilizing different energy functions.

To prove this theorem, we first establish some basic properties of linear attractors.
Consider a fixed point $z^*$ to the generic ODE \eqref{eq:general_ODE}. Suppose the ODE is locally Lipschitz continuous around $z^*$, i.e.,  there exists $H>0$ such that
\begin{equation}\label{eq:smoothness_ode}
    \|\nabla G(z)\| \le H \text{ for any } z\in \{z|\|z-z^*\|\le 1\} \ .
\end{equation}
The following proposition presents a sufficient condition for whether a fixed point is a linear attractor:
\begin{prop}\label{thm:general_sufficient_condition}
    Consider a stationary point $z^*$ to the dynamic~\eqref{eq:general_ODE}, namely $G(z^*)=0$. For any positive definite matrice $P\succ 0$, if
    \begin{equation}\label{eq:sufficient_condi_general}
    \frac{1}{2}\pran{\nabla G(z^*)^T P + P\nabla G(z^*)} \prec 0 \ , 
    \end{equation}
    then $z^*$ is a linear attractor to the dynamic $\dot z = G(z)$ in the $P$ norm.
\end{prop}


The next proposition shows that the above sufficient condition is essentially tight. The only slack between these results is boundary case when the Jacobian-like term $\frac{1}{2}\pran{\nabla G(z^*)^T P + P\nabla G(z^*)}$ is only negative semidefinite, in which case high-order derivatives will determine whether the point is attractive. 


\begin{prop}\label{prop:general_necessary_condition}
    Consider a fixed point $z^*$ to the dynamic \eqref{eq:general_ODE}, namely $G(z^*)=0$. For any positive definite matrix $P\succ 0$, if
    \begin{equation}\label{eq:necessary}
    \lambda_{max}\left(\frac{1}{2}\pran{\nabla G(z^*)^T P + P\nabla G(z^*)}\right) >0 \ , 
    \end{equation}
    then $z^*$ is not a linear attractor to the dynamic $\dot z = G(z)$ in the $P$ norm.
\end{prop}

Applying these propositions to the four continuous-time dynamics we consider is the core of our proof of Theorem~\ref{thm:main}.
One key step in establishing these conditions is choosing the corresponding norms for studying different algorithms. As we can see in Figure~\ref{fig:cifar}, the trajectory of AGDA follows from a slightly skewed ellipsoid while that of GDA is more symmetric. This is fundamentally because GDA has simultaneous updates in primal and dual, while AGDA utilizes sequential updates. This difference results in different choices of norm $P$ when studying the two algorithms. It turns out the natural norms for GF, GDA and EGM are the same: the Euclidean norm with $P=I$, and the natural norm for AGDA is an scaled Euclidean norm with $P=\twomatrix{I}{\frac{1}{2} sB^T}{\frac{1}{2} sB}{I}$. These norms are obtained by examining the dynamics of each method when solving bilinear systems $L(x,y)=y^T B x$.

\begin{remark}
It holds for bilinear problem $L(x,y)=y^T B x$ that the LHS of \eqref{eq:sufficient_condition_AGDA} is $0$. This explains the circling over an ellipsoid behavior of AGDA when solving the bilinear problem. Indeed, the ellipsoid is given by the matrix $P=\twomatrix{I}{\frac{1}{2} sB^T}{\frac{1}{2} sB}{I}$. Also note that not all norms are equivalent in terms of linear attractors, indeed AGDA may linearly attract in the $P$-norm but fail to monotonically decrease the Euclidean norm.
\end{remark}

\subsection{Local Convergence of the Discrete-time Algorithms}\label{sec:discrete}
Here we extend our guarantees from Theorem \ref{thm:main} from the underlying ODEs to the discrete-time algorithms. Similar to the results stated in Section \ref{sec:local_ODE}, we first present general results and then apply them to the specific algorithms.

We consider a generic discrete-time algorithm with iterate update $z^+=g(z,s)$ and the dynamic around a fixed point $z^*$. As defined in \citet{Lu2020}, we say the $O(s)$-resolution ODE is gradient-based if for any $\delta>0$, there exists a constant $c>0$ such that it holds for any $z\in\{\|F(z)\|\le \delta\}$ and a small enough stepsize $s$ that
\begin{equation}\label{eq:gradient-based-ODE}
    \|z(s)-z^+\|\le c s^{3}  \|F(z)\| \ ,
\end{equation}
where $z^+=g(z,s)$, and $z(s)$ is the solution obtained at $t=s$ following its $O(s)$-resolution ODE with initial solution $z(0)=z$. Indeed, the $O(s)$-resolution ODE of AGDA, EGM, GDA are gradient-based when $L(x,y)$ is smooth enough. This is formalized in the follow proposition utilizing  fifth-order differentiablity, which is a fairly mild condition as in practice many objective function is analytical (i.e. infinitely differentiable), for example, any smooth loss with sigmoid/logistic activation function for GANs.

\begin{prop}\label{prop:gradient-based}
    Suppose $L(x,y)$ is fifth-order differentiable over $x$ and $y$ and $\nabla^j F(z)$ is bounded for $j=1,2,3,4$. 
    Then $O(s)$-resolution ODE of AGDA, EGM, GDA are gradient-based.
\end{prop}

Using this condition, we find that if the $z^*$ is a linear attractor to the $O(s)$-resolution ODE, then it is also a linear attractor to the discrete-time algorithm.








\begin{theorem}\label{thm:general-dta}
    Consider a discrete-time algorithm with iterate update $z^+=g(z,s)$, and its $O(s)$-resolution ODE \eqref{eq:general_ODE}. Suppose the $O(s)$-resolution ODE is gradient-based. Suppose $z^*$ is a linear attractor to its $O(s)$-resolution ODE and $$\lambda_{max}\left(\frac{1}{2}\pran{\nabla G(z^*)^T P + P\nabla G(z^*)}\right)\ge b s\ ,$$
    with $b>0$.
    Then there exists $s^*$, such that for any $s\le s^*$, $z^*$ is a linear attractor to the discrete-time algorithm.
\end{theorem}

As a direct consequence of Theorem \ref{thm:general-dta}, we know that the conditions \eqref{eq:sufficient_condition_GDA}, \eqref{eq:sufficient_condition_PPM} and \eqref{eq:sufficient_condition_AGDA} are sufficient conditions for $z^*$ being an attractive for each of GDA, AGDA and EGM when the stepsize $s$ is reasonably small (i.e., smaller than a constant). Lastly, we note that when the stepsize $s$ is reasonably small, the stationary points of the $O(s)$-resolution ODEs we have considered so far are indeed also stationary points of $L(x,y)$, completing our recovery from ODE guarantees back to the discrete-time algorithms themselves.
\begin{prop} \label{prop:stationary}
    When the stepsize $s$ is sufficiently small, the stationary points to the $O(s)$-resolution ODE of GDA, EGM and AGDA, namely \eqref{eq:ODE-GDA}, \eqref{eq:ODE-PPM} and \eqref{eq:ODE-AGDA} are also stationary points to the general minimax problem \eqref{eq:minimax}.
\end{prop}

\section{Phase Transitions between Limit Points and Limiting Cycles} \label{sec:hopf}
The GAN problem considered in Section~\ref{sec:GAN} demonstrated that divergence, cycling, and convergence all easily arise from a simple formulation of GAN training. Meanwhile,  Figure~\ref{fig:GDA-regularization} showed that adding mild amounts of regularization sufficed to cause the trajectories of the considered methods to all transition to convergence. Here we explain these phase transitions in the discrete algorithmic iterations using their related ODEs. 
In continuous-time dynamical systems, the transformation from an (attractive or repulsive) limit point into a limit cycle (or vice versa) can occur when the eigenvalues of the Jacobian $\nabla G(z)$ transition from being all negative real part to having some nonnegative real values. When exactly one conjugate pair of eigenvalues crosses to having positive real part, the dynamics of the system are described as a {\it Hopf Bifurcation}.

Such transitions are controlled by this pair of eigenvectors and hence are fundamentally a two-dimensional phenomena. This notion can be generalized to multi-dimension problems using the Central Manifold Theorem~\citep{Kelley1967}, which essentially state that one can locally decompose the multi-dimension spaces into independently evolving dynamics on one-dimension and/or two-dimension manifold. Thus, one could reparameterize the dynamics $\dot z=G(z)$ to isolate these behaviors, however, the details of such an approach are beyond the scope of this work.

\begin{figure*}\centering
    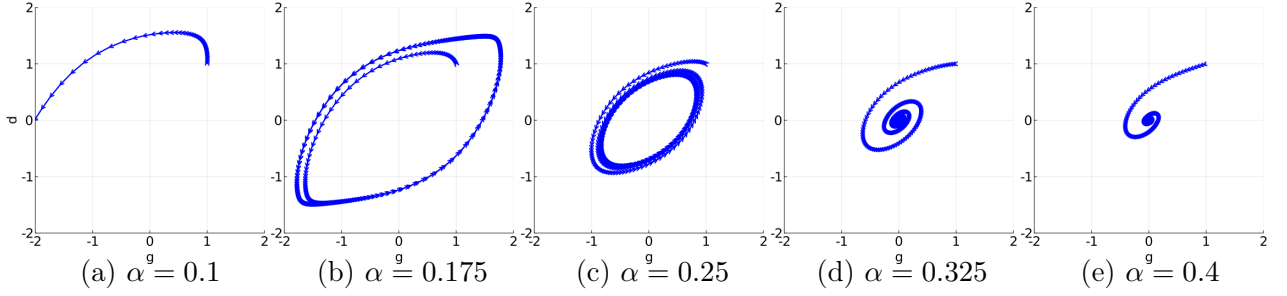
    \caption{Regularization causing GDA trajectory on the GAN~\eqref{eq:regularized} to undergo a supercritical Hopf bifurcation.}
    \label{fig:gd}
\end{figure*}

To formalize the phase transition of solution path dynamics as a hyperparameter $\alpha\in\RR$ changes, we consider families of minimax problems
$$ \min_x \max_y L(x,y,\alpha).$$
For example, this captures our previously considered setting of regularizing our one-layer GAN setup~\eqref{eq:regularized} on CIFAR-10 data. The following theorem guarantees for a wide class of Guassian data and latent distributions applying GDA to~\eqref{eq:regularized} has a repulsive fixed point that transitions to become attractive with sufficient regularization.
\begin{theorem}\label{thm:gaussian}
    For any covariance matrices $\Sigma\succeq \Sigma'$ in $p_{data} \sim N(0,\Sigma)$ and $p_{latent} \sim N(0,\Sigma')$, the origin is a stationary point of~\eqref{eq:regularized} which is repulsive to the GDA dynamics~\eqref{eq:ODE-GDA} for $\alpha=0$ and transitions to be attractive whenever
    $$ \alpha \geq \frac{\lambda_{max}(\Sigma -\Sigma')}{4} + O(s) \ . $$
\end{theorem}

In the rest of this section, we consider concrete examples of Hopf bifurcations arising. For ease of presentation, we consider a two-dimension system that corresponds to a conjugate pair of eigenvalues crossing from negative to positive real part (potentially coming from reparameterizing to isolate the independently evolving two-dimension manifold). We suppose the stationary point of interest occurs at the origin $z=0$ and the dynamics $\dot z=G(z)$ at the critical $\alpha^*$ are in the form
\begin{equation} \label{eq:diag-form}
    \begin{bmatrix} \dot x \\ \dot y \end{bmatrix} = \begin{bmatrix} 0 & -w \\ w & 0\end{bmatrix}\begin{bmatrix} x \\ y \end{bmatrix} + \begin{bmatrix} P(x,y) \\ Q(x,y) \end{bmatrix}
\end{equation}
where $w\neq 0$ and $P,Q$ both have $P(0,0)=Q(0,0)=0$ and $\nabla P(0,0)=\nabla Q(0,0)=0$. Indeed, most of the dynamic around the stationary point $z=0$ and critical point $\alpha^*$ can be transformed to \eqref{eq:diag-form} except for some degenerated cases~\citep{Kuznetsov9801}.
For dynamics of the form~\eqref{eq:diag-form}, the shape of the Hopf bifurcation occurring for some critical $\alpha^*$ is determined by the {\it First Lyapunov Coefficient}, defined as
\begin{align} 
   l_1(0) = &\frac{1}{8w}\left(P_{xxx}+P_{xyy} + Q_{xxy}+Q_{yyy}\right)\nonumber\\
   & +\frac{1}{8w^2}(P_{xy}(P_{xx}+P_{yy})-Q_{xy}(Q_{xx}+Q_{yy})\nonumber\\
   & \hskip2cm -P_{xx}Q_{xx} +P_{yy}Q_{yy}) \label{eq:lyapunov formula}
\end{align}
where subscripts denote partial derivatives at $(0,0)$.

Depending on the sign of this coefficient, two types of bifurcations can arise, {\it supercritical} and {\it subcritical}. A supercritical bifurcation (occuring when $l_1(0)<0$) corresponds to an attractive limit cycle transforming into an attractive limit point and a subcritical bifurcation (occuring when $l_1(0)>0$) corresponds to a repulsive limit point transforming into an attractive limit point surrounded by a repulsive limit cycle, respectively.

\subsection{GAN Hopf Bifurcations}
First, we illustrate a supercritical Hopf bifurcation by considering a two-dimensional case of the GAN formulation~\eqref{eq:regularized} fixing $p_{data}$ and $p_{latent}$ to both produce $\pm 1$ with probability one half each. Consequently, if the generator plays $g=0$, it exactly matches the true data distribution and is the desired solution. However, GDA fails to converge at all for this problem unless sufficient regularization is added. Figure~\ref{fig:gd} shows GDA with $s=0.2$ and full gradient evaluations diverging or cycling when $\alpha \leq 0.25$ and converging when $\alpha\geq 0.325$ (much like we previously saw in Figure~\ref{fig:GDA-regularization}).

Numerically, we see that the transition between cycling and convergence happens around $ \alpha^* \approx 0.26872 $. To verify this is a supercritical Hopf bifurcation, we simply need to write the GDA dynamics~\eqref{eq:GDA} in the form~\eqref{eq:diag-form} and then compute the Lyapunov coefficient. One valid reparameterization of the GDA dynamics on $(g,d)$ comes from considering
$$  \begin{bmatrix} x \\ y \end{bmatrix} \approx \begin{bmatrix} 0.57735 & 0.57735 \\ -1 & 1 \end{bmatrix} \begin{bmatrix} g \\ d \end{bmatrix} $$
which for our example problem has 
\begin{equation*}
    \begin{bmatrix} \dot x \\ \dot y \end{bmatrix} \approx \begin{bmatrix} 0 & -0.43463 \\ 0.43463 & 0\end{bmatrix}\begin{bmatrix} x \\ y \end{bmatrix} + o(\|(x,y)\|).
\end{equation*}
Then computing higher derivatives of these dynamics\footnote{At the critical $\alpha^*\approx 0.26872 $, our reparameterization has $ w \approx 0.434633 $,
$ P_{xxx} \approx -0.20695 $,
$ P_{xyy} \approx  0.01227 $,
$ Q_{xxy} \approx -0.20695 $, and
$ Q_{yyy} \approx  0.01227 $} verifies $l_1(0) \approx -0.11198 <0$ and so indeed the phenomena in Figure~\ref{fig:gd} is a supercritical Hopf bifurcation.

\begin{figure*}\centering
    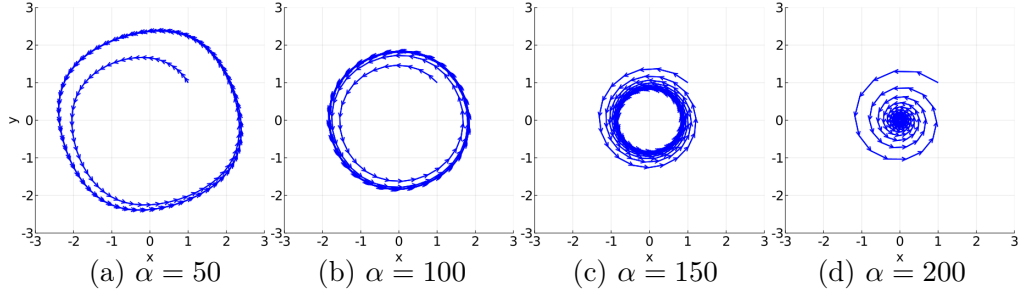
    \caption{EGM trajectory on the polynomial~\eqref{eq:W} undergoing a supercritical Hopf bifurcation as $\alpha$ increases.}
    \label{fig:W}
\end{figure*}
\begin{figure*}\centering
    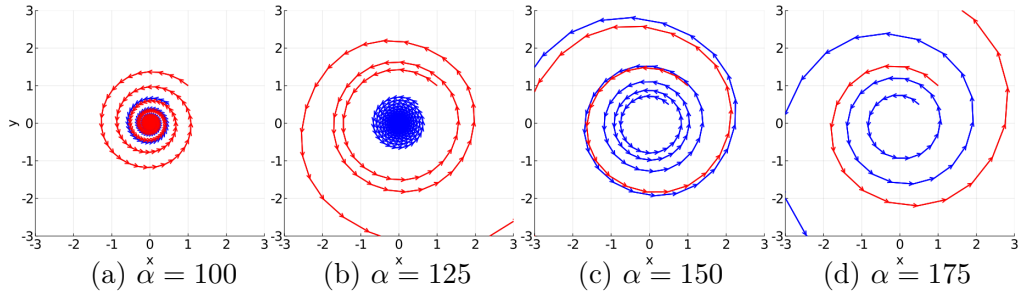
    \caption{GDA trajectories on~\eqref{eq:M} from two different initializations are shown. These dynamics undergo a subcritical Hopf bifurcation as a repulsive limit cycle collapses into the origin transforming it to from attractive to repulsive.}
    \label{fig:M}
\end{figure*}

\subsection{Polynomial Supercritical Bifurcation}

Consider the degree four polynomial objective function
\begin{equation}\label{eq:W}
   \min_x\max_y\ f(x) + \alpha xy - f(y) 
\end{equation}
where $ f(x) = (x+3)(x+1)(x-1)(x-3) $.
The behavior of EGM with $(x_0,y_0)=(1,1)$ and $s=0.002$ as $\alpha$ varies is shown in Figure~\ref{fig:W}. As $\alpha$ increases past $\alpha^* \approx 165$, the trajectories attractive limit cycle contracts towards the origin, converting the origin from a repulsive limit point to an attractive one.

Considering the related ODE dynamics~\eqref{eq:ODE-PPM} at $\alpha^*=\sqrt{20400}\approx 143$, the dynamics undergo a supercritical Hopf Bifirucation, which matches the transition we observed in Figure~\ref{fig:W}.
To verify this, we need to verify that these dynamics are of the form~\eqref{eq:diag-form} and compute the associated first Lyapunov coefficient. This computation is detailed in the appendix, finding that 
$w\approx 148.542$ and $l_1(0) \approx -0.04362$
and consequently, this phase transition corresponds to a supercritical bifurcation.

The mild difference between the discrete-time iteration's critical $\alpha^*$ around $165$ and the ODE's critical point around $143$ is not unexpected as the ODE only approximates the behavior of the first-order method to an accuracy $o(s^2)$. Higher fidelity ODEs could be used following closer to the discrete-time iteration, reducing this difference in when the bifurcation occurs.


\subsection{Polynomial Subcritical Bifurcation}

In a subcritical Hopf bifurcation , as $\alpha$ varies through $\alpha^*$, a repulsive limit cycle collapses into being a repulsive limit point (or in reverse, a repulsive limit point splits into being a repulsive limit cycle, around an attractive limit point). This behavior can be observed on a degree four polynomial objective function. Consider applying GDA to
\begin{equation}\label{eq:M} 
    \min_x\max_y\ g(x) +\alpha xy - g(y)
\end{equation}
where $ g(x) = -(x+3)(x+1)(x-1)(x-3) $.
As $\alpha$ decreases past $\alpha^*\approx 135$, the origin changes from being repulsive to attractive, surrounded by a repulsive cycle. Further decreasing $\alpha$ grows this repulsive cycle, creating a larger area attracted to the origin. Figure~\ref{fig:M} shows trajectories initialized at $(1/2,1/2)$ and $(1,1)$ which are excluded from this attractive region by $\alpha=125$ and $\alpha=150$, respectively.

Considering the related ODE dynamics~\eqref{eq:ODE-GDA} at $\alpha^*=\sqrt{20400}\approx 143$, the dynamics undergo a subcritical Hopf bifirucation, verified in the appendix as $w\approx 148.542$ and $l_1(0) \approx 0.04362$.
Since $l_1(0)$ is positive, this phase transition indeed corresponds to a subcritical bifurcation.

\section{Conclusion}
Inspired by the topologically different limiting behavior seen in GAN training, we derived necessary and sufficient conditions for a stationary point to be attractive for each of GDA, AGDA, and EGM that avoid sending the stepsize parameter to zero. 
These results are based understanding the ODE dynamics related to each method, which further allows us to describe the observed phase transitions to convergence as Hopf bifurcations. This explains the transitions observed when adding regularization to our simplified family of GAN training problems.


\bibliographystyle{apalike} 
\bibliography{references}

\begin{thebibliography}{}

\bibitem[Adolphs et~al., 2019]{adolphs1901}
Adolphs, L., Daneshmand, H., Lucchi, A., and Hofmann, T. (2019).
\newblock Local saddle point optimization: A curvature exploitation approach.
\newblock volume~89 of {\em Proceedings of Machine Learning Research}, pages
  486--495. PMLR.

\bibitem[Arora et~al., 2017]{Arora1710}
Arora, S., Ge, R., Liang, Y., Ma, T., and Zhang, Y. (2017).
\newblock Generalization and equilibrium in generative adversarial nets
  ({GAN}s).
\newblock volume~70 of {\em Proceedings of Machine Learning Research}, pages
  224--232, International Convention Centre, Sydney, Australia. PMLR.

\bibitem[Ben-Tal et~al., 2009]{ben2009robust}
Ben-Tal, A., El~Ghaoui, L., and Nemirovski, A. (2009).
\newblock {\em Robust optimization}, volume~28.
\newblock Princeton University Press.

\bibitem[Bertsimas et~al., 2011]{bertsimas2011theory}
Bertsimas, D., Brown, D.~B., and Caramanis, C. (2011).
\newblock Theory and applications of robust optimization.
\newblock {\em SIAM review}, 53(3):464--501.

\bibitem[Cherukuri et~al., 2016]{Cherukuri1611}
Cherukuri, A., Gharesifard, B., and Cortes, J. (2016).
\newblock Saddle-point dynamics: conditions for asymptotic stability of saddle
  points.
\newblock arXiv: 1510.02145.

\bibitem[Daskalakis and Panageas, 2018]{Constantinos1812}
Daskalakis, C. and Panageas, I. (2018).
\newblock The limit points of (optimistic) gradient descent in min-max
  optimization.
\newblock In {\em Proceedings of the 32nd International Conference on Neural
  Information Processing Systems}, NIPS'18, page 9256–9266, Red Hook, NY,
  USA. Curran Associates Inc.

\bibitem[Goodfellow et~al., 2014]{Goodfellow1412}
Goodfellow, I.~J., Pouget-Abadie, J., Mirza, M., Xu, B., Warde-Farley, D.,
  Ozair, S., Courville, A., and Bengio, Y. (2014).
\newblock Generative adversarial nets.
\newblock In {\em Proceedings of the 27th International Conference on Neural
  Information Processing Systems - Volume 2}, NIPS'14, page 2672–2680,
  Cambridge, MA, USA. MIT Press.

\bibitem[Grimmer et~al., 2020]{Grimmer2006}
Grimmer, B., Lu, H., Worah, P., and Mirrokni, V. (2020).
\newblock The landscape of the proximal point method for nonconvex-nonconcave
  minimax optimization.
\newblock arXiv: 2006.08667.

\bibitem[Hsieh et~al., 2020]{Hsieh2006}
Hsieh, Y.-P., Mertikopoulos, P., and Cevher, V. (2020).
\newblock The limits of min-max optimization algorithms: convergence to
  spurious non-critical sets.
\newblock arXiv: 2006.09065.

\bibitem[Jin et~al., 2020]{Jordan2008}
Jin, C., Netrapalli, P., and Jordan, M.~I. (2020).
\newblock What is local optimality in nonconvex-nonconcave minimax
  optimization?
\newblock arXiv: 1902.00618.

\bibitem[Kelley, 1967]{Kelley1967}
Kelley, A. (1967).
\newblock The stable, center-stable, center, center-unstable, unstable
  manifolds.
\newblock {\em Journal of Differential Equations}, 3(4):546 -- 570.

\bibitem[Kuznetsov, 1998]{Kuznetsov9801}
Kuznetsov, Y.~A. (1998).
\newblock {\em Elements of Applied Bifurcation Theory (2nd Ed.)}.
\newblock Springer-Verlag, Berlin, Heidelberg.

\bibitem[Letcher, 2020]{Letcher2005}
Letcher, A. (2020).
\newblock On the impossibility of global convergence in multi-loss
  optimization.
\newblock arXiv: 2005.12649.

\bibitem[Liang and Stokes, 2019]{liang2019interaction}
Liang, T. and Stokes, J. (2019).
\newblock Interaction matters: A note on non-asymptotic local convergence of
  generative adversarial networks.
\newblock In {\em The 22nd International Conference on Artificial Intelligence
  and Statistics}, pages 907--915.

\bibitem[Lin et~al., 2018]{Lin1809}
Lin, Q., Liu, M., Rafique, H., and Yang, T. (2018).
\newblock Solving weakly-convex-weakly-concave saddle-point problems as
  successive strongly monotone variational inequalities.
\newblock arXiv: 1810.10207.

\bibitem[Lu, 2020]{Lu2020}
Lu, H. (2020).
\newblock An $o(s^r)$-resolution ode framework for discrete-time optimization
  algorithms and applications to the linear convergence of minimax problems.
\newblock arXiv: 2001.08826.

\bibitem[Mazumdar et~al., 2020]{Mazumdar2002}
Mazumdar, E., Ratliff, L.~J., and Sastry, S.~S. (2020).
\newblock On gradient-based learning in continuous games.
\newblock {\em SIAM Journal on Mathematics of Data Science}, 2(1):103–131.

\bibitem[Nouiehed et~al., 2019]{Nouiehed1902}
Nouiehed, M., Sanjabi, M., Huang, T., Lee, J.~D., and Razaviyayn, M. (2019).
\newblock Solving a class of non-convex min-max games using iterative first
  order methods.
\newblock In {\em Advances in Neural Information Processing Systems 32}, pages
  14934--14942. Curran Associates, Inc.

\bibitem[Shi et~al., 2018]{Shi1810}
Shi, B., Du, S.~S., Jordan, M.~I., and Su, W.~J. (2018).
\newblock Understanding the acceleration phenomenon via high-resolution
  differential equations.
\newblock arXiv: 1810.08907.

\bibitem[Su et~al., 2016]{Su1601}
Su, W., Boyd, S., and Cand{{\`e}}s, E.~J. (2016).
\newblock A differential equation for modeling nesterov's accelerated gradient
  method: Theory and insights.
\newblock {\em Journal of Machine Learning Research}, 17(153):1--43.

\bibitem[Verdu and Poor, 1984]{verdu1984minimax}
Verdu, S. and Poor, H. (1984).
\newblock On minimax robustness: A general approach and applications.
\newblock {\em IEEE Transactions on Information Theory}, 30(2):328--340.

\bibitem[Yang et~al., 2020]{Yang2002}
Yang, J., Kiyavash, N., and He, N. (2020).
\newblock Global convergence and variance-reduced optimization for a class of
  nonconvex-nonconcave minimax problems.
\newblock arXiv: 2002.09621.

\bibitem[Zhang et~al., 2020]{Zhang2020}
Zhang, G., Poupart, P., and Yu, Y. (2020).
\newblock Optimality and stability in non-convex smooth games.

\end{thebibliography}

\appendix

\section{More trajectories on simple CIFAR GANs}\label{app:batch}

The trajectories presented in Figure~\ref{fig:cifar} and~\ref{fig:GDA-regularization} all use a moderate batch size of $10$, reducing stochastic effects. Figure~\ref{fig:batch} presents $500$ steps of a sample trajectory when this batch size is $1$, $10$, or $100$ for each method. Largely, the dynamics of GDA, AGDA, and EGM stay the same. GDA still diverges and AGDA still converges into an attractive limit cycle but with much greater variance in how closely it follows it. EGM still converges towards the origin, eventually staying in a ball around the origin whose radius shrinks as the variance in the gradient samples shrinks.

\begin{figure*}[h]
    \centering
    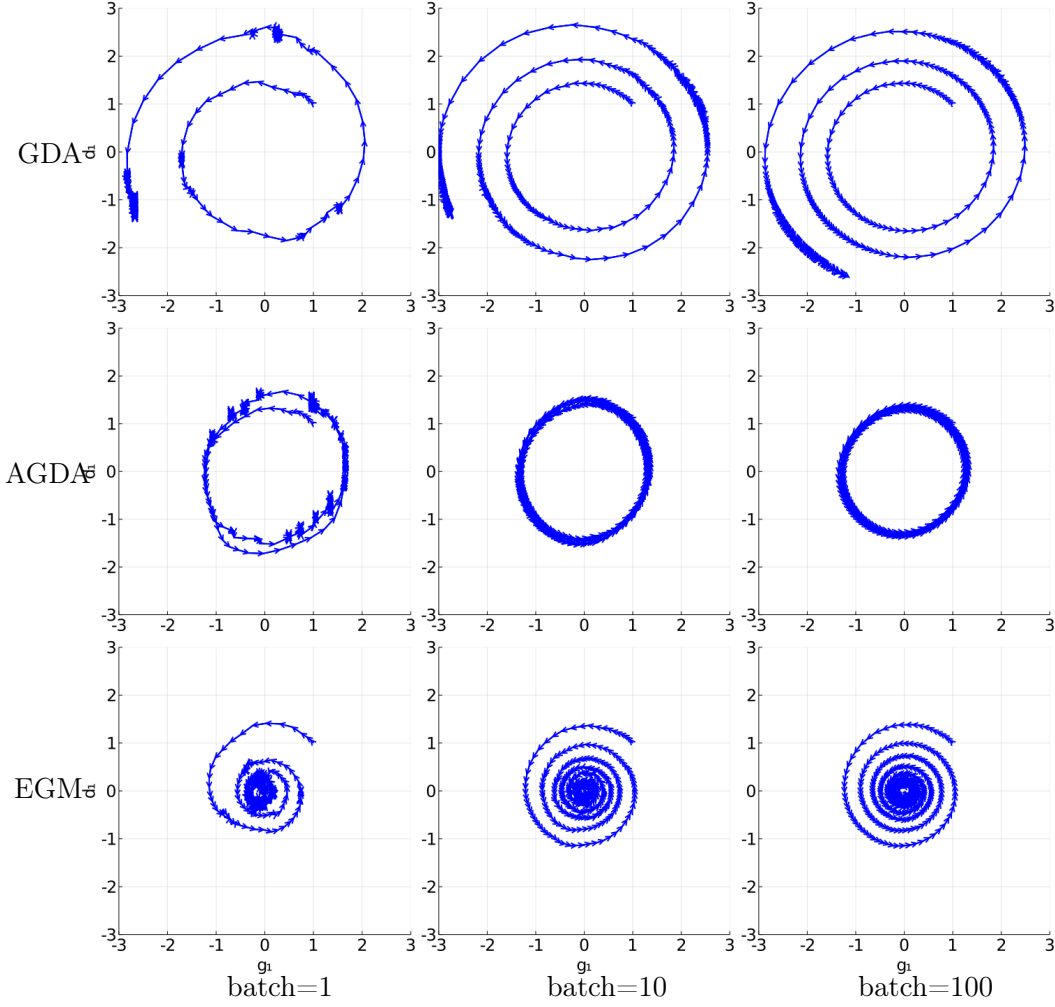
    \caption{Sample trajectories with batch sizes $1$, $10$, and $100$ for GDA, AGDA, and EGM on our simplified GAN over CIFAR data.}
    \label{fig:batch}
\end{figure*}

\section{Deferred Proofs} \label{sec:proofs}

\subsection{Proof of Proposition \ref{prop:ODEs}}

    We utilize Theorem 1 in \citet{Lu2020} and the notations therein to compute the $O(s)$-resolution ODE.
    
    For GDA and EGM, their $O(s)$-resolution ODEs are presented in Corollary 1 in \citet{Lu2020}.
    
    For AGDA, it follows by Taylor expansion of the update rule over stepsize $s$ that
    \begin{align}\label{eq:AGDA-expansion}
        z_{k+1}= z_k + s g_1(z_k) + s^2 g_2(z_k) + s^3 g_3(z_k)  + o(s^3)\ ,
    \end{align}
    where
    $ g_1(z_k)=-F(z_k)$,
    $g_2(z_k)=\vectorr{0}{\nabla_x F_y(z_k)F_x(z_k)}$, and $g_3(z_k)=\vectorr{0}{-\nabla_{xx}F_y(z_k)(F_x(z_k), F_x(z_k))}$
    with $\nabla_{xx}F_y(z_k)$ being a tensor. 
    Now suppose the $O(s)$-resolution ODE of AGDA is $\dz = f_0(z)+sf_1(z)$, then it follows by Theorem 1 in \citet{Lu2020} that
    $$f_0(z)=g_1(z)=-F(z)$$
    $$f_1(z)=g_2(z)-\frac{1}{2}\nabla f_0(z) f_0(z)=-\frac{1}{2}\twomatrix{A}{B^T}{B}{C} F(z) \ .$$
    Hence \eqref{eq:ODE-AGDA} is the $O(s)$-resolution ODE of AGDA. 

\subsection{Proof of Proposition \ref{thm:general_sufficient_condition}}

    We here claim the following fact, which directly proves the proposition:
    Suppose there exists positive definite matrices $\alpha>0$ such that
    \begin{equation}\label{eq:sufficient_condition}
    \frac{1}{2}\pran{\nabla G(z^*)^T P + P\nabla G(z^*)} \preceq -\alpha I \ , 
    \end{equation}
    Let $\uP=\lambda_{\max}(P)$ and $\lP=\lambda_{\min}(P)$ be the maximal and minimal eigenvalue of p.s.d. matrix $P$. Suppose $H>0$ satisfies that $\|\nabla^2 G(z)\|\le H$ for $z\in \{z|\|z-z^*\|_P\le 1\}$ 
    and the initial point $z(0)$ is close enough to $z^*$: $\|z(0) - z^*\|_P \leq \min\{\frac{\alpha \sqrt{\lP}}{H\uP}, 1\}\ ,$ 
    then the trajectory $z(t)$ monotonically converges to $z^*$ with
    \begin{equation}\label{eq:linear-decay}
        \|z(t) - z^*\|^2_P \leq e^{-{\alpha t/\uP}}\|z(0)-z^*\|^2_P.
    \end{equation}
    Consider the Lyapunov function $\frac{1}{2}\|z(t) - z^*\|^2_P$. This is monotonically decreasing as 
    \begin{align*}
        \frac{\partial}{\partial t} \frac{1}{2}\|z(t) - z^*\|^2_P &= (z(t)-z^*)^TP G(z(t))\\
        &\le (z(t)-z^*)^TP\nabla G(z^*)(z(t)-z^*) + \frac{H\uP}{2}\|z(t)-z^*\|^3\\
        &= \frac{1}{2} (z(t)-z^*)^T(\nabla G(z^*)^TP + P\nabla G(z^*))(z(t)-z^*) + \frac{H\uP}{2}\|z(t)-z^*\|^3\\
        &\leq -\alpha\|z(t)-z^*\|^2 + \frac{H\uP}{2\sqrt{\lP}}\|z(t)-z^*\|^2 \|z(t)-z^*\|_P \\
        & \le -\frac{\alpha}{2}\|z(t)-z^*\|^2 \\
        & \le -\frac{\alpha}{2\uP}\|z(t)-z^*\|^2_P \ ,
    \end{align*}
    where the first inequality utilizes $\|\nabla G(z)\|\le H$, the second inequality is from \eqref{eq:sufficient_condition} and $\|z(t)-z^*\|\le \frac{1}{\sqrt{\lP}}\|z(t)-z^*\|_P$, and the third inequality follows from the fact that $\|z(t)-z^*\|_P\le \frac{\alpha\sqrt{\lP}}{H\uP}$ by noticing $\|z(t)-z^*\|_P$ is monotonically non-increasing over $t$. \eqref{eq:linear-decay} follows immediately from the above inequality.

\subsection{Proof of Proposition \ref{prop:general_necessary_condition}}
By condition \eqref{eq:necessary}, we know that there exists some $\alpha>0$ such that
$$\lambda_{max}\left(\frac{1}{2}\pran{\nabla G(z^*)^T P + P\nabla G(z^*)}\right) \ge \alpha I\ .$$
Let $e$ to be the unit eigenvector of $\pran{\nabla G(z^*)^T P + P\nabla G(z^*)}$ that correspond to the max eigenvalue. For any $\delta>0$, consider taking $z(0)= z^*+\min\{\frac{\alpha}{H\uP}, \delta\} e$. Then it holds that 
\begin{align*}
        \frac{\partial}{\partial t} \frac{1}{2}\|z(0) - z^*\|^2_P &= (z(0)-z^*)^TP G(z(0))\\
        &\ge (z(0)-z^*)^TP\nabla G(z^*)(z(0)-z^*) - \frac{H\uP}{2}\|z(0)-z^*\|^3\\
        &= \frac{1}{2} (z(0)-z^*)^T(\nabla G(z^*)^TP + P\nabla G(z^*))(z(0)-z^*) - \frac{H\uP}{2}\|z(0)-z^*\|^2 \\
        &\ge \alpha\|z(0)-z^*\|^2 - \frac{H\uP}{2}\|z(0)-z^*\|^3 \\
        & \ge \frac{\alpha}{2}\|z(0)-z^*\|^2 \\
        & > 0 \ ,
    \end{align*}
    where the first inequality utilizes $\|\nabla G(z)\|\le H$, the second inequality is from the definition of $z(0)$, and the third inequality follows from the fact that $\|z(0)-z^*\|\le \frac{\alpha}{H\uP}$. Therefore, for a smaller enough $t$, it holds that $\|z(t) - z^*\|_P>\|z(0) - z^*\|_P$, which contradicts with the definition of begin a linear attractor.

\subsection{Proof of Theorem \ref{thm:main}}
 1. For Gradient Flow \eqref{eq:ODE-GF}, we take $P=I$. Notice that 
\begin{align*}
    \frac{1}{2}\pran{\nabla G(z^*)^T  + \nabla G(z^*)} = -\frac{1}{2}\pran{\nabla F(z^*)+ (\nabla F(z^*))^T} = \twomatrix{-A}{}{}{-C}\prec 0 \ ,
\end{align*}
where the inequality is due to \eqref{eq:sufficient_condition_GF}. Applying Propositions \ref{thm:general_sufficient_condition} and \ref{prop:general_necessary_condition} with $P=I$ finishes the proof.

2. For Gradient Descent Ascent \eqref{eq:ODE-GDA}, we take $P=I$. Notice that 
{\small
\begin{align*}
    & \frac{1}{2}\pran{\nabla G(z^*)^T  + \nabla G(z^*)} \\ 
    = & -\frac{1}{2}\pran{\nabla F(z^*)+ (\nabla F(z^*))^T} - \frac{s}{4} \pran{(\nabla F(z^*))^2 + ((\nabla F(z^*))^2)^T} - \frac{s}{4} \pran{\nabla^3 F(z^*) F(z^*) + (\nabla^3 F(z^*) F(z^*))^T} \\
    = & -\frac{1}{2}\pran{\nabla F(z^*)+ (\nabla F(z^*))^T} - \frac{s}{4} \pran{(\nabla F(z^*))^2 + ((\nabla F(z^*))^2)^T} \\
    = & \twomatrix{-A - \frac{s}{2} \pran{A^2-B^T B}}{}{}{-C - \frac{s}{2} \pran{C^2-B B^T}} \\
    \prec & 0 \ ,
\end{align*}
}
where the second equality is from $F(z^*)=0$, the third equality uses the definition of $\nabla F(z^*)$ by noticing $\nabla F(z^*)$ and $(\nabla F(z^*))^2$ are both generalized block skew-symmetric, the first inequality utilizes $\|A\|, \|-C\| \le \|F(z^*)\| \le 1/s $, and the last inequality is due to \eqref{eq:sufficient_condition_GDA}. Applying Propositions \ref{thm:general_sufficient_condition} and \ref{prop:general_necessary_condition} with $P=I$ finishes the proof.

3. Similar to 2., we have for EGM \eqref{eq:ODE-PPM} that
{\small
\begin{align*}
    & \frac{1}{2}\pran{\nabla G(z^*)^T  + \nabla G(z^*)} \\ 
    = & -\frac{1}{2}\pran{\nabla F(z^*)+ (\nabla F(z^*))^T} + \frac{s}{4} \pran{(\nabla F(z^*))^2 + ((\nabla F(z^*))^2)^T} + \frac{s}{4} \pran{\nabla^3 F(z^*) F(z^*) + (\nabla^3 F(z^*) F(z^*))^T} \\
    = & -\frac{1}{2}\pran{\nabla F(z^*)+ (\nabla F(z^*))^T} + \frac{s}{4} \pran{(\nabla F(z^*))^2 + ((\nabla F(z^*))^2)^T} \\
    = & \twomatrix{-A + \frac{s}{2} \pran{A^2-B^T B}}{}{}{-C + \frac{s}{2} \pran{C^2-B B^T}} \\
    \prec & 0 \ ,
\end{align*}
}
where the last inequality is due to \eqref{eq:sufficient_condition_PPM}. Applying Propositions \ref{thm:general_sufficient_condition} and \ref{prop:general_necessary_condition} with $P=I$ finishes the proof.

4.  For Alternating Gradient Descent Ascent \eqref{eq:ODE-AGDA}, we take $P=\twomatrix{I}{\frac{1}{2} sB^T}{\frac{1}{2} sB}{I}$, and consider $M=\frac{1}{2} \pran{P\nabla G(z^*) + \nabla G(z^*)^T P}$. Then $M$ is symmetric and
$$M_{xx}=-A-\frac{s}{2}A^2 - \frac{s^2}{4}\pran{A B^T B + B^T B A} ,$$
$$M_{yy}=-C-\frac{s}{2}C^2 - \frac{s^2}{4}\pran{C B B^T + B B^T C} ,$$
$$M_{xy}=-\frac{s}{2}\pran{BA+CB} - \frac{s^2}{4}\pran{BA^2 + C^2 B} . $$
Applying Propositions \ref{thm:general_sufficient_condition} and \ref{prop:general_necessary_condition} finishes the proof.

\subsection{Proof of Proposition \ref{prop:gradient-based}}

    Corollary 2 in \citet{Lu2020} shows that under the condition of Proposition \ref{prop:gradient-based}, the $O(s)$-resolution ODEs of EGM and GDA are gradient-based. We here just show that the $O(s)$-resolution ODE of AGDA is gradient-based. Recall the definition of $g_1,g_2,g_3$ in \eqref{eq:AGDA-expansion}. Notice that under the condition of the proposition, we have $\|g_j(z)\|\le O(\|F(z)\|)$ and $\|\nabla^k g_j(z)\| \le O(1)$ for $j=1,2,3$ and $k=1,...,5-j$. Thus it follows from Theorem 3 in \citet{Lu2020} that the $O(s)$-resolution ODE of AGDA is also gradient-based.

\subsection{Proof of Theorem \ref{thm:general-dta}}

    Suppose $\|z-z^*\|\le \delta:= \min\{\frac{bs\sqrt{\lP}}{H\uP},1\}$. Let $z(s)$ be the solution of the $O(s)$-resolution ODE at $t=s$ from initial solution $z(0)=z$ and $z^+=g(z,s)$. Then it follows by the proof of Proposition \ref{thm:general_sufficient_condition} that $\|z(s)-z^*\|_P\le e^{-\frac{1}{2} bs^2/\uP}\|z-z^*\|_P$. Moreover, it follows from the definition of gradient-based $O(s)$-resolution ODE that there exists $c>0$ such that $\|z(s)-z^+\|\le cs^3\|F(z)\|$. Now notice $F(z)$ is Lipschtiz continuous in the close and bounded set $B(z^*, \delta)$, i.e., there exists $\gamma>0$ such that $\|F(z)\|=\|F(z)-F(z^*)\|\le \gamma\|z-z^*\|$. Putting everything together and letting $s^*=\frac{b\sqrt{\lP}}{8c\gamma \uP^2}$, we have for any $s\le s^*$ that
    \begin{align*}
        &\|z^+-z^*\|_P\\
        & \le \|z(s)-z^*\|_P + \|z^+ - z(s)\|_P \\
        & \le e^{-\frac{1}{2}b s^2/\uP}\|z-z^*\|_P + \uP\|z^+ - z(s)\| \\
        & \le e^{-\frac{1}{2}b s^2/\uP}\|z-z^*\|_P + \uP cs^3\|F(z)\| \\
        & \le \pran{1-\frac{1}{4}b s^2/\uP}\|z-z^*\|_P + \uP cs^3 \gamma \|z-z^*\| \\
        & \le \pran{1-\frac{1}{4}b s^2/\uP}\|z-z^*\|_P + \frac{\uP}{\sqrt{\lP}} cs^3 \gamma \|z-z^*\|_P \\
        & \le \pran{1-\frac{1}{8}\frac{b s^2}{\uP}}\|z-z^*\|_P \ ,
    \end{align*}
    where the second inequality utilizes Proposition \ref{thm:general_sufficient_condition} and the definition of $\uP$, the third inequality utilizes \eqref{eq:gradient-based-ODE}, the fourth inequality utilizes $\|F(z)\|=\|F(z)-F(z^*)\|\le \gamma \|z-z^*\|$ and the last inequality is because $s\le s^*$.
    This proves the theorem by telescoping.

\subsection{Proof of Proposition \ref{prop:stationary}}

    For GDA, we have $G(z)=-\pran{I+\frac{s}{2}\nabla F(z)} F(z)$. Notice for small enough $s$, it holds that $\|\frac{s}{2}\nabla F(z)\|\le \frac{1}{2}$, whereby $I+\frac{s}{2}\nabla F(z)$ is a full-rank matrix. Therefore, $G(z)=0$ is equivalent to $F(z)=0$, which finishes the proof for GDA. The proof for AGDA and EGM follow the same process.

\subsection{Proof of Theorem~\ref{thm:gaussian}}
	The dynamics we are considering are given by
	$$ \dot z = -F(z) - \frac{s}{2}\nabla F(z) F(z) \ . $$
	The objective we are considering is
	$$ L(g,d) = \mathbb{E}_{s\sim N(0,I)} \left[\log \left(\frac{1}{1+e^{d^T\Sigma^{1/2}s}}\right) + \log\left(1-\frac{1}{1+e^{d^T(\Sigma'^{1/2}s+g)}}\right)\right] + \frac{\alpha}{2}\|g\|^2 - \frac{\alpha}{2}\|d\|^2 $$
	which has gradient given by 
	$$ F(g,d) = \begin{bmatrix} \nabla_g L(g,d) \\ -\nabla_d L(g,d)
	\end{bmatrix} = \begin{bmatrix} \mathbb{E}_{s\sim N(0,I)} \left[ \frac{1}{1+e^{d^T(\Sigma'^{1/2}s+g)}} d \right] + \alpha g \\ \mathbb{E}_{s\sim N(0,I)} \left[\frac{e^{d^T\Sigma^{1/2}s}}{1+e^{d^T\Sigma^{1/2}s}}\Sigma^{1/2}s -  \frac{1}{1+e^{d^T(\Sigma'^{1/2}s+g)}} (\Sigma'^{1/2}s+g) \right] + \alpha d 
	\end{bmatrix} \ . $$
	Evaluating this at $(g,d)=(0,0)$ verifies that the origin is stationary for any choice of $\alpha\in\RR$ as
	$$ F(0,0) = \begin{bmatrix}0 \\ \mathbb{E}_{s\sim N(0,I)} \left[\frac{e^{0}}{1+e^{0}}\Sigma^{1/2}s -  \frac{1}{1+e^{0}} \Sigma'^{1/2}s \right]
	\end{bmatrix} =  \begin{bmatrix} 0 \\ 0
	\end{bmatrix} $$
	where the last equality relies on our consideration of zero mean data distributions.
	Next we evaluate the Jacobian of $F(g,d)$ to understand the related continuous dynamics giving
	$$ \nabla_g F(g,d) = \begin{bmatrix} \mathbb{E}_{s\sim N(0,I)} \left[ \frac{-e^{d^T(\Sigma'^{1/2}s+g)}}{(1+e^{d^T(\Sigma'^{1/2}s+g)})^2} dd^T \right] + \alpha I \\ \mathbb{E}_{s\sim N(0,I)} \left[ \frac{e^{d^T(\Sigma'^{1/2}s+g)}}{(1+e^{d^T(\Sigma'^{1/2}s+g)})^2} (\Sigma'^{1/2}s+g)d^T - \frac{1}{1+e^{d^T(\Sigma'^{1/2}s+g)}} I \right] 
	\end{bmatrix} $$
	and
	$$ \nabla_d F(g,d) = \begin{bmatrix} \mathbb{E}_{s\sim N(0,I)} \left[ \frac{-e^{d^T(\Sigma'^{1/2}s+g)}}{(1+e^{d^T(\Sigma'^{1/2}s+g)})^2} d(\Sigma'^{1/2}s+g)^T + \frac{1}{1+e^{d^T(\Sigma'^{1/2}s+g)}} I \right]  \\ \mathbb{E}_{s\sim N(0,I)} \left[\frac{e^{d^T\Sigma^{1/2}s}}{(1+e^{d^T\Sigma^{1/2}s})^2}\Sigma^{1/2}ss^T\Sigma^{1/2} -\frac{e^{d^T(\Sigma'^{1/2}s+g)}}{(1+e^{d^T(\Sigma'^{1/2}s+g)})^2} (\Sigma'^{1/2}s+g)(\Sigma'^{1/2}s+g)^T \right] + \alpha I
	\end{bmatrix} $$
	Consequently, the Gradient Flow (GF) dynamics at the origin have
	$$ -\nabla F(0,0) = \begin{bmatrix} -\alpha I & \frac{1}{2} I \\ \frac{-1}{2} I & \frac{1}{4}(\Sigma -\Sigma')- \alpha I
	\end{bmatrix} \ . $$

	The dynamics of GDA are given by
	\begin{align*}
		\dot g = &\mathbb{E}_{s\sim N(0,I)} \left[ \frac{-1}{1+e^{d^T(\Sigma'^{1/2}s+g)}} d \right] - \alpha g\\
		&+\frac{s}{2}\Bigg(\left(\mathbb{E}_{s\sim N(0,I)} \left[ \frac{-e^{d^T(\Sigma'^{1/2}s+g)}}{(1+e^{d^T(\Sigma'^{1/2}s+g)})^2} dd^T \right] + \alpha I\right)\left(\mathbb{E}_{s\sim N(0,I)} \left[ \frac{-1}{1+e^{d^T(\Sigma'^{1/2}s+g)}} d \right] - \alpha g\right)\\
		& \hskip1.2cm +\left(\mathbb{E}_{s\sim N(0,I)} \left[ \frac{-e^{d^T(\Sigma'^{1/2}s+g)}}{(1+e^{d^T(\Sigma'^{1/2}s+g)})^2} d(\Sigma'^{1/2}s+g)^T + \frac{1}{1+e^{d^T(\Sigma'^{1/2}s+g)}} I \right]\right)\\
		&\hskip1.7cm \cdot \left(\mathbb{E}_{s\sim N(0,I)} \left[\frac{-e^{d^T\Sigma^{1/2}s}}{1+e^{d^T\Sigma^{1/2}s}}\Sigma^{1/2}s +  \frac{1}{1+e^{d^T(\Sigma'^{1/2}s+g)}} (\Sigma'^{1/2}s+g) \right] - \alpha d \right)\Bigg)
	\end{align*}
	and
	\begin{align*}
		\dot d = &\mathbb{E}_{s\sim N(0,I)} \left[\frac{-e^{d^T\Sigma^{1/2}s}}{1+e^{d^T\Sigma^{1/2}s}}\Sigma^{1/2}s +  \frac{1}{1+e^{d^T(\Sigma'^{1/2}s+g)}} (\Sigma'^{1/2}s+g) \right] - \alpha d \\
		&+\frac{s}{2}\Bigg(\left(\mathbb{E}_{s\sim N(0,I)} \left[ \frac{e^{d^T(\Sigma'^{1/2}s+g)}}{(1+e^{d^T(\Sigma'^{1/2}s+g)})^2} (\Sigma'^{1/2}s+g)d^T - \frac{1}{1+e^{d^T(\Sigma'^{1/2}s+g)}} I \right]\right)\\
		&\hskip1.7cm \cdot \left(\mathbb{E}_{s\sim N(0,I)} \left[ \frac{-1}{1+e^{d^T(\Sigma'^{1/2}s+g)}} d \right] - \alpha g\right)\\
		& \hskip1.2cm +\left(\mathbb{E}_{s\sim N(0,I)} \left[\frac{e^{d^T\Sigma^{1/2}s}}{(1+e^{d^T\Sigma^{1/2}s})^2}\Sigma^{1/2}ss^T\Sigma^{1/2} -\frac{e^{d^T(\Sigma'^{1/2}s+g)}}{(1+e^{d^T(\Sigma'^{1/2}s+g)})^2} (\Sigma'^{1/2}s+g)(\Sigma'^{1/2}s+g)^T \right] + \alpha I\right)\\
		&\hskip1.7cm \cdot \left(\mathbb{E}_{s\sim N(0,I)} \left[\frac{-e^{d^T\Sigma^{1/2}s}}{1+e^{d^T\Sigma^{1/2}s}}\Sigma^{1/2}s +  \frac{1}{1+e^{d^T(\Sigma'^{1/2}s+g)}} (\Sigma'^{1/2}s+g) \right] - \alpha d \right)\Bigg) \ .
	\end{align*}
	
	Taking the Jacobian of these more complicated GDA dynamics at $(g,d)=(0,0)$ provides an $O(s)$ correction to the GF dynamics above as the following
	$$  \nabla G(0,0) = \begin{bmatrix} -\left(\alpha + \frac{s}{2}\left(\alpha^2-\frac{1}{4}\right)\right) I &  \frac{1}{2} I +\frac{s}{16}\Sigma'  \\ \frac{-1}{2} I -\frac{s}{16}\Sigma' &  \frac{1}{4}(\Sigma -\Sigma')-\alpha I +\frac{s}{2}\left( \frac{1}{4}I +  (\frac{1}{4}(\Sigma -\Sigma')+\alpha I)(\frac{1}{4}(\Sigma -\Sigma')-\alpha I)\right)
\end{bmatrix} \ . $$
	
	Next we verify that the origin is not attractive for $\alpha=0$. This is immediate as the Jacobian above simplifies to the following very structured block-matrix
	$$  \nabla G(0,0) = \begin{bmatrix} \frac{s}{8} I &  \frac{1}{2} I +\frac{s}{16}\Sigma'  \\ \frac{-1}{2} I -\frac{s}{16}\Sigma' &  \frac{1}{4}(\Sigma -\Sigma') +\frac{s}{2}\left( \frac{1}{4}I +  \frac{1}{16}(\Sigma -\Sigma')^2\right)
\end{bmatrix} \ . $$
	Since we assume $\Sigma \succeq \Sigma'$, the diagonal blocks on this matrix are both positive definite and so this matrix has at least one eigenvalue with positive real part.
	
	Finally we verify that the origin transitions to being attractive for all sufficiently large $\alpha$. This follows from observing that the diagonal entries are all decreasing $O(-\alpha^2)$ while the off-diagonal entries have magnitude $O(\alpha)$. Then applying Gershgorin's circle theorem guarantees that once the diagonal is sufficiently negative, all of the eigenvalues of GDA's dynamics must be negative (and thus the origin becomes a linear attractor). 


\section{Calculation of First Lyapunov Coefficients} \label{app:computations}
We give a proposition providing formulas for $l_1(0)$ for symmetric bilinear problems.
Using this formula, we can compute the first Lyapunov coefficient for both polynomial examples considered in Section~\ref{sec:hopf}, verifying the phase transitions observed are supercritical and subcritical respectively.
\begin{lemma}\label{lem:hopf-formulas}
    Consider any symmetric and bilinear minimax problem $\min_{x\in \RR}\max_{y\in \RR} L(x,y,\alpha)= f(x)+\alpha xy-f(y)$ with a stationary point at $(x,y)=(0,0)$. Then the dynamics of GF~\eqref{eq:ODE-GF} at the origin have First Lyapunov Coefficient given by
    \begin{align*}
        w &= \alpha,\\
        l_1(0) &= -\frac{f^{(4)}(0)}{4\alpha}.
    \end{align*} 
    The dynamics of GDA~\eqref{eq:ODE-GDA} at the origin have First Lyapunov Coefficient given by
    \begin{align*}
        w &= \alpha(1+sf^{(2)}(0)),\\
        l_1(0) &= -\frac{f^{(4)}(0)}{4\alpha(1+sf^{(2)}(0))} -s\frac{4f^{(4)}(0)f^{(2)}(0)+3f^{(3)}(0)^2}{8\alpha(1+sf^{(2)}(0))} +\frac{2sf^{(3)}(0)^2+3s^2f^{(3)}(0)^2f^{(2)}(0)}{16\alpha(1+sf^{(2)}(0))^2}.
    \end{align*}
    The dynamics of EG~\eqref{eq:ODE-PPM} at the origin have First Lyapunov Coefficient given by
    \begin{align*}
        w &= \alpha(1-sf^{(2)}(0)),\\
        l_1(0) &= -\frac{f^{(4)}(0)}{4\alpha(1-sf^{(2)}(0))} +s\frac{4f^{(4)}(0)f^{(2)}(0)+3f^{(3)}(0)^2}{8\alpha(1-sf^{(2)}(0))}-\frac{2sf^{(3)}(0)^2-3s^2f^{(3)}(0)^2f^{(2)}(0)}{16\alpha(1-sf^{(2)}(0))^2}.
    \end{align*}
\end{lemma}
\begin{proof}
    First, we observe that for any problem with $L(x,y,0)=f(x)+axy-f(y)$, each of our algorithms give dynamics of the type~\eqref{eq:diag-form}. The symmetry between $x$ and $y$ in the objective function and in EGM and GDA ensures that whenever the trace of $\nabla G(z)$ is zero, the two diagonal entries of $\nabla G(z)$ must be zero.

    The dynamics of GF~\eqref{eq:ODE-GF} for this class of symmetric problems are given by
    \begin{equation*}
        \begin{bmatrix} \dot x \\ \dot y \end{bmatrix} = \begin{bmatrix} -f'(x) - \alpha y \\ \alpha x-f'(y) \end{bmatrix}
    \end{equation*}   
    From this, the following derivatives in terms of the form~\eqref{eq:diag-form} are given by
    \begin{align*}
        w = Q_x &= \alpha, \\
        P_{xx}=Q_{yy} &=-f^{(3)}(0), \\
        P_{xy}=P_{yy}=Q_{xy}=Q_{xx} &=0, \\
        P_{xxx}=Q_{yyy} &=-f^{(4)}(0), \\
        P_{xyy}=Q_{xxy} &=0.
    \end{align*}
    Then the claimed formula follows from~\eqref{eq:lyapunov formula}.

    The dynamics of GDA~\eqref{eq:ODE-GDA} for this class of symmetric problems are given by
    \begin{equation*}
        \begin{bmatrix} \dot x \\ \dot y \end{bmatrix} = \begin{bmatrix} -f'(x) - \alpha y \\ \alpha x-f'(y) \end{bmatrix} - \frac{s}{2}\begin{bmatrix} f^{(2)}(x)(f'(x) + \alpha y) -\alpha(\alpha x-f'(y)) \\ \alpha(f'(x)+\alpha y) + f^{(2)}(y)(\alpha x-f'(y)) \end{bmatrix}
    \end{equation*}   
    From this, the following derivatives in terms of the form~\eqref{eq:diag-form} are given by
    \begin{align*}
        w = Q_x &= a + s\alpha f^{(2)}(0), \\
        P_{xx}=Q_{yy} &=-f^{(3)}(0) - \frac{3sf^{(3)}(0)f^{(2)}(0)}{2}, \\
        P_{yy}=-Q_{xx} &= \frac{-s\alpha f^{(3)}(0)}{2}, \\
        P_{xy}=Q_{xy} &=0, \\
        P_{xxx}=Q_{yyy} &=-f^{(4)}(0) - \frac{s}{2}\left(4f^{(4)}(0)f^{(2)}(0)+3f^{(3)}(0)^2\right), \\
        P_{xyy}=Q_{xxy} &=0.
    \end{align*}
    Then the claimed formula follows from~\eqref{eq:lyapunov formula}.
    
    The dynamics of EGM~\eqref{eq:ODE-PPM} for this class of symmetric problems are given by
    \begin{equation*}
        \begin{bmatrix} \dot x \\ \dot y \end{bmatrix} = \begin{bmatrix} -f'(x) - \alpha y \\ \alpha x-f'(y) \end{bmatrix} + \frac{s}{2}\begin{bmatrix} f^{(2)}(x)(f'(x) + \alpha y) -\alpha(\alpha x-f'(y)) \\ \alpha(f'(x)+\alpha y) + f^{(2)}(y)(\alpha x-f'(y)) \end{bmatrix}
    \end{equation*}   
    From this, the following derivatives in terms of the form~\eqref{eq:diag-form} are given by
    \begin{align*}
        w = Q_x &= a - s\alpha f^{(2)}(0), \\
        P_{xx}=Q_{yy} &=-f^{(3)}(0) + \frac{3sf^{(3)}(0)f^{(2)}(0)}{2}, \\
        P_{yy}=-Q_{xx} &= \frac{s\alpha f^{(3)}(0)}{2}, \\
        P_{xy}=Q_{xy} &=0, \\
        P_{xxx}=Q_{yyy} &=-f^{(4)}(0) + \frac{s}{2}\left(4f^{(4)}(0)f^{(2)}(0)+3f^{(3)}(0)^2\right), \\
        P_{xyy}=Q_{xxy} &=0. 
    \end{align*}
    Then the claimed formula follows from~\eqref{eq:lyapunov formula}.
\end{proof}

For the example~\eqref{eq:W}, the function $ f(x) = (x+3)(x+1)(x-1)(x-3) $ has
\begin{align*}
    f^{(4)}(0)&=24\\
    f^{(3)}(0)&=0\\
    f^{(2)}(0)&=-20\\
    f^{(1)}(0)&=0.
\end{align*}
Plugging this into Lemma~\ref{lem:hopf-formulas} yields the EGM ODE dynamics at $\alpha^*=\sqrt{20400}$ having
\begin{align*}
        w &= \alpha(1-sf^{(2)}(0))=\sqrt{20400}(1-0.002\times-20) \approx 148.542,\\
        l_1(0) &=-\frac{f^{(4)}(0)}{4\alpha(1-sf^{(2)}(0))} +s\frac{4f^{(4)}(0)f^{(2)}(0)+3f^{(3)}(0)^2}{8\alpha(1-sf^{(2)}(0))}-\frac{2sf^{(3)}(0)^2-3s^2f^{(3)}(0)^2f^{(2)}(0)}{16\alpha(1-sf^{(2)}(0))^2} \approx -0.04362.
\end{align*}
Likewise, for~\eqref{eq:M}, the function $ g(x) = -(x+3)(x+1)(x-1)(x-3) $ has
\begin{align*}
    f^{(4)}(0)&=-24\\
    f^{(3)}(0)&=0\\
    f^{(2)}(0)&=20\\
    f^{(1)}(0)&=0.
\end{align*}
Plugging this into Lemma~\ref{lem:hopf-formulas} yields the GDA ODE dynamics at $\alpha^*=\sqrt{20400}$ having
\begin{align*}
        w &= \alpha(1+sf^{(2)}(0))=\sqrt{20400}(1+0.002\times20) \approx 148.542,\\
        l_1(0) &=-\frac{f^{(4)}(0)}{4\alpha(1+sf^{(2)}(0))} -s\frac{4f^{(4)}(0)f^{(2)}(0)+3f^{(3)}(0)^2}{8\alpha(1+sf^{(2)}(0))} +\frac{2sf^{(3)}(0)^2+3s^2f^{(3)}(0)^2f^{(2)}(0)}{16\alpha(1+sf^{(2)}(0))^2} \approx 0.04362.
\end{align*}

\end{document}